\documentclass[preprint,12pt]{elsarticle}

\usepackage{graphicx}

\usepackage{amsmath,bm}
\usepackage{amssymb}
\usepackage{amsthm}
\newtheorem{thm}{Theorem}
\newtheorem{lem}[thm]{Lemma}
\newtheorem{cor}[thm]{Corollary}

\makeatletter
\DeclareRobustCommand*\CFrac[2]{%
  \mathinner{\mathchoice
    {\@CFrac\textstyle\textfont{4}{#1}{#2}}%
    {\@CFrac\scriptstyle\textfont{}{#1}{#2}}%
    {\@CFrac\scriptscriptstyle\scriptfont{}{#1}{#2}}%
    {\@CFrac\scriptscriptstyle\scriptscriptfont{}{#1}{#2}}}}
\def\@CFrac#1#2#3#4#5{%
  {\dimen\tw@\fontdimen8#23\relax \dimen\tw@#3\dimen\tw@
    \setbox\z@\hbox{\m@th$#1\;{#4}\;\mathstrut$}%
    \@tempdima\dp\z@ \advance\@tempdima\dimen\tw@
    \setbox\tw@\hbox{\m@th$#1\;{#5}\;\mathstrut$}%
    \@tempdimb\ht\tw@ \advance\@tempdimb\dimen\tw@
    \dimen@ \ifdim\wd\z@>\wd\tw@ \wd\z@ \else \wd\tw@ \fi
    \setbox\z@\hbox to\dimen@{%
      \hss\unhbox\z@\hss
      \vrule \vrule\@width\z@\@depth\@tempdima}%
    \setbox\tw@\hbox to\dimen@{%
      \vrule\@width\z@\@height\@tempdimb \vrule
      \hss\unhbox\tw@\hss}%
    \dimen\tw@\fontdimen22#22\relax
    \advance\dimen\tw@\ht\z@ \advance\dimen\tw@\dp\z@
    \advance\dimen\tw@ .2\p@
    \setbox\z@\vbox{\box\z@ \hrule \box\tw@}%
    \advance\dimen\tw@-\ht\z@
    \raise\dimen\tw@\box\z@}}
\makeatother

\journal{}

\begin{document}

\begin{frontmatter}



\title{Laurent biorthogonal polynomials, $q$-Narayana polynomials and domino tilings of the Aztec diamonds}


\author{Shuhei Kamioka}

\ead{kamioka.shuhei.3w@kyoto-u.ac.jp}

\address{Department of Applied Mathematics and Physics, Graduate School of Informatics,
  Kyoto University, Kyoto 606-8501, Japan}

\begin{abstract}
  A T\"oplitz determinant whose entries are described
  by a $q$-analogue of the Narayana polynomials is evaluated
  by means of Laurent biorthogonal polynomials which allow of a combinatorial interpretation
  in terms of Schr\"oder paths.
  As an application, a new proof is given to the Aztec diamond theorem
  by Elkies, Kuperberg, Larsen and Propp concerning domino tilings of the Aztec diamonds.
  The proof is based on the correspondence with non-intersecting Schr\"oder paths developed by Eu and Fu.
\end{abstract}


\end{frontmatter}


\section{Introduction}
\label{sec:introduction}

{\em Laurent biorthogonal polynomials (LBPs)} are orthogonal functions which play fundamental roles
in the theory of two-point Pad\'e approximants at zero and infinity \cite{Jones-Thron(1982)}.
In Pad\'e approximants, LBPs appear as the denominators of the convergents of a T-fraction.
(See also, e.g., \cite[Chapter 7]{Jones-Thron(1980CF)}
and \cite{Hendriksen-VanRossum(1986),Zhedanov(1998)}.)
Recently, the author exhibited a combinatorial interpretation of LBPs
in terms of lattice paths called {\em Schr\"oder paths} \cite{Kamioka(2007),Kamioka(2008)}.
In this paper, we utilize LBPs to calculate a determinant
whose entries are given by a $q$-analogue of the Narayana polynomials \cite{Bonin-Shapiro-Simion(1993)}
which have a combinatorial expression in Schr\"oder paths.
As an application, we give a new proof to the Aztec diamond theorem
by  Elkies, Kuperberg, Larsen and Propp
\cite{Elkies-Kuperberg-Larsen-Propp(1992.01),Elkies-Kuperberg-Larsen-Propp(1992.02)}
by means of LBPs and Schr\"oder paths.

A {\em Schr\"oder path} $P$ is a lattice path in the two-dimensional plane $\mathbb{Z}^{2}$
consisting of up steps $(1,1)$, down steps $(1,-1)$ and level steps $(2,0)$,
and never going beneath the $x$-axis.
See Figure \ref{fig:SchPath} for example.
\begin{figure}
  \centering
\includegraphics{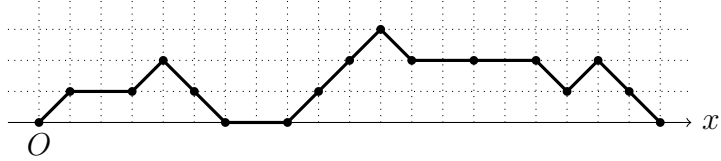}
  \caption{A Schr\"oder path $P \in S_{10}$ such that $\mathrm{level}(P) = 4$ and $\mathrm{area}(P) = 20$.}
  \label{fig:SchPath}
\end{figure}
For $k \in \mathbb{N} = {\{ 0,1,2,\ldots \}}$,
let $S_k$ denote the set of Schr\"oder paths from $(0,0)$ to $(2k,0)$.
The number $\# S_k$ of such paths is counted by the $k$-th large Schr\"oder number (A006318 in OEIS \cite{OEIS}).
The first few of $\# S_k$ are $1$, $2$, $6$, $22$ and $90$.

Enumerative or statistical properties of Schr\"oder paths are often investigated through the {\em Narayana polynomials}
\begin{gather}
  L_{k}(t) = \sum_{j=1}^{k} \frac{1}{k} \binom{k}{j} \binom{k}{j-1} (1+t)^{j}, \qquad k \in \mathbb{N},
\end{gather}
where $N_0(t) = 1$.
(The coefficients $\frac{1}{k} \binom{k}{j} \binom{k}{j-1}$ are the {\em Narayana numbers},
A001263 in OEIS \cite{OEIS}.)
Bonin, Shapiro and Simion \cite{Bonin-Shapiro-Simion(1993)} interpreted the Narayana polynomials
by counting the level steps in Schr\"oder paths,
\begin{gather}
  L_{k}(t) = \sum_{P \in S_k} t^{\mathrm{level}(P)}
\end{gather}
where $\mathrm{level}(P)$ denotes the number of level steps in a Schr\"oder path $P$.
(Level steps in this paper are identified with ``diagonal'' steps in \cite{Bonin-Shapiro-Simion(1993)}.)
For more about the Narayana polynomials and related topics,
see, e.g., Sulanke's paper \cite{Sulanke(2002)} and the references therein.
Besides level steps, Bonin et al.~also examined the area polynomials
\begin{gather}
  A_{k}(q) = \sum_{P \in S_k} q^{\mathrm{area}(P)}, \qquad k \in \mathbb{N},
\end{gather}
with respect to the statistic $\mathrm{area}(P)$ that measures the area bordered by a path $P$ and the $x$-axis.
(In \cite{Bonin-Shapiro-Simion(1993)}, the major index is also examined, but we will not consider in this paper.)
In this paper, we consider the two statistics $\mathrm{level}(P)$ and $\mathrm{area}(P)$ simultaneously
in the polynomials
\begin{gather}
  N_{k}(t,q) = \sum_{P \in S_k} t^{\mathrm{level}(P)} q^{\mathrm{area}(P)}, \qquad k \in \mathbb{N}.
\end{gather}
We refer to $N_{k}(t,q)$ by the {\em $q$-Narayana polynomials}.
Obviously, the $q$-Narayana polynomials satisfy that $N_k(t,1) = L_k(t)$ and $N_k(1,q) = A_k(t)$,
and reduce to the large Schr\"oder numbers, $N_k(1,1) = \# S_k$,
as well as the Catalan numbers, $N_k(0,1) = \frac{1}{k+1} \binom{2k}{k}$.
In Section \ref{sec:qNarPolMom}, we find the LBPs of which the moments are described by the $q$-Narayana polynomials.

The aim of this paper is twofold:
(i) to calculate a determinant whose entries are described by the $q$-Narayana polynomials $N_k(t,q)$;
(ii) to give a new proof to the Aztec diamond theorem by Elkies, Kuperberg, Larsen and Propp
\cite{Elkies-Kuperberg-Larsen-Propp(1992.01),Elkies-Kuperberg-Larsen-Propp(1992.02)}
by means of LBPs and Schr\"oder paths.

Determinants whose entries are given
by the large Schr\"oder numbers, by the Narayana polynomials and by their $q$-analogues are calculated
by many authors using various techniques.
Ishikawa, Tagawa and Zeng \cite{Ishikawa-Tagawa-Zeng(2009)} found a closed-form expression
of Hankel determinants of a $q$-analogue of the large Schr\"oder numbers
in a combinatorial way based on Gessel--Viennot's lemma \cite{Gessel-Viennot(1985)}.
Petkovi\'c, Barry and Rajkovi\'c \cite{Petkovic-Barry-Rajkovic(2012)} calculated Hankel determinants
described by the Narayana polynomials
using an analytic method of solving a moment problem of orthogonal polynomials.
In Section \ref{sec:ToeplitzDets},
we evaluate a T\"oplitz determinant described by the $q$-Narayana polynomials $N_k(t,q)$
by means of a combinatorial interpretation of LBPs in terms of Schr\"oder paths.

Counting domino tilings of the Aztec diamonds is a typical problem of tilings which is exactly solvable.
For $n \in \mathbb{N}$, the {\em Aztec diamond $\mathit{AD}_{n}$ of order $n$} is the union
of all unit squares which lie inside the closed region $|x|+|y| \le n+1$.
A domino denotes a one-by-two or two-by-one rectangle.
Then, a {\em domino tiling}, or simply a {\em tiling}, of $\mathit{AD}_{n}$ is a collection
of non-overlapping dominoes which exactly covers $\mathit{AD}_{n}$.
Figure \ref{fig:AD_tiling} shows an Aztec diamond and an example of a tiling.
\begin{figure}
  \centering
\includegraphics{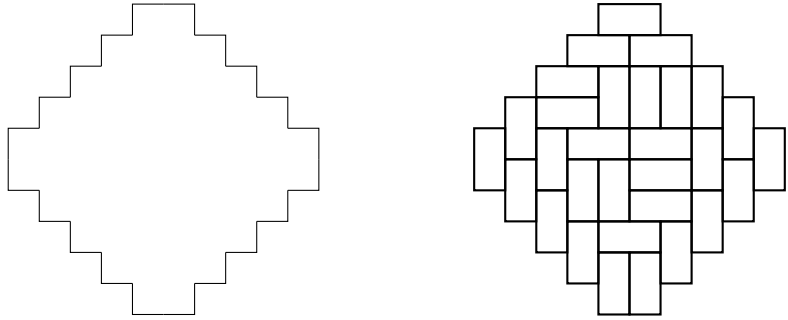}
  \caption{The Aztec diamond $\mathit{AD}_{5}$ (left) and a tiling of $\mathit{AD}_{5}$ (right).}
  \label{fig:AD_tiling}
\end{figure}
Let $T_n$ denote the set of all tilings of $\mathit{AD}_{n}$.
Elkies, Kuperberg, Larsen and Propp,
in their two-parted paper \cite{Elkies-Kuperberg-Larsen-Propp(1992.01),Elkies-Kuperberg-Larsen-Propp(1992.02)},
considered the statistics $v(T)$ and $r(T)$ of a tiling $T$,
where $v(T)$ denotes half the number of vertical dominoes in $T$ and $r(T)$ the {\em rank} of $T$.
(The definition of the rank is explained in Section \ref{sec:ADT}.)
They showed that the counting polynomials
\begin{gather}
  \mathrm{AD}_{n}(t,q) = \sum_{T \in T_n} t^{v(T)} q^{r(T)}, \qquad n \in \mathbb{N},
\end{gather}
admit the following closed-form expression.

\begin{thm}[Aztec diamond theorem
    \cite{Elkies-Kuperberg-Larsen-Propp(1992.01),Elkies-Kuperberg-Larsen-Propp(1992.02)}]
  \label{thm:ADT}
  For $n \in \mathbb{N}$,
  \begin{gather} \label{eq:ADT}
    \mathrm{AD}_{n}(t,q) = \prod_{k=0}^{n-1} (1 + t q^{2k+1})^{n-k}.
  \end{gather}
\end{thm}

Especially, the number $\# T_n$ of possible tilings of $\mathit{AD}_{n}$ equals to
\begin{gather} \label{eq:ADT11}
  \# T_n = \mathrm{AD}_{n}(1,1) = 2^{\frac{n(n+1)}{2}}.
\end{gather}
(That is the solution to Exercise 6.49b in Stanley's book \cite{Stanley(1999EC2)}.)
In \cite{Elkies-Kuperberg-Larsen-Propp(1992.01),Elkies-Kuperberg-Larsen-Propp(1992.02)},
a proof by means of the domino shuffling is shown for \eqref{eq:ADT}
as well as three different proofs for \eqref{eq:ADT11}.
Further different proofs for \eqref{eq:ADT11} are given by several authors
\cite{Ciucu(1996),Kuo(2004),Brualdi-Kirkland(2005),Eu-Fu(2005)}.
In particular, Eu and Fu \cite{Eu-Fu(2005)} gives a proof of \eqref{eq:ADT11}
by calculating Hankel determinants of the large and the small Schr\"oder numbers.
They showed a one-to-one correspondence between tilings and tuples of non-intersecting Schr\"oder paths
to apply Gessel--Viennot's lemma \cite{Gessel-Viennot(1985),Aigner(2001LNCS)}
on non-intersecting paths and determinants.
In this paper, we give a new proof to \eqref{eq:ADT} based on the correspondence developed by Eu and Fu.
Clarifying the connection between the statistics $v(T)$ and $r(T)$ of tilings $T$
and the statistics $\mathrm{level}(P)$ and $\mathrm{area}(P)$ of Schr\"oder paths $P$,
we reduce the proof to the calculation of a determinant of the $q$-Narayana polynomials.

This paper is organized as follows.
In Section \ref{sec:LBPs_TFrac}, we recall the definitions and the fundamentals of LBPs and T-fractions
focusing on the moments and a moment determinant of the T\"oplitz form.
Sections \ref{sec:momPaths}--\ref{sec:nonIntPaths} concern a combinatorial interpretation of LBPs
in terms of Schr\"oder paths which is applicable to general families of LBPs.
In Section \ref{sec:momPaths},
we exhibit a combinatorial expression of the moments of LBPs (Theorem \ref{thm:momPaths})
with two different proofs.
In Section \ref{sec:nonIntPaths},
we show a combinatorial expression of the moment determinant in terms of non-intersecting Schr\"oder paths
(Theorem \ref{thm:detInNIPaths}) based on Gessel--Viennot's methodology \cite{Gessel-Viennot(1985),Aigner(2001LNCS)}.

Sections \ref{sec:qNarPolMom}--\ref{sec:ADT} concern the special case of the moments of LBPs
given by the $q$-Narayana polynomials.
In Section \ref{sec:qNarPolMom}, we find the LBPs whose moments are given by the $q$-Narayana polynomials
(Theorem \ref{thm:NarayanaPolsInMoms}).
In Section \ref{sec:ToeplitzDets}, we evaluate a determinant of the $q$-Narayana polynomials
by calculating the moment determinant of LBPs (Theorem \ref{thm:DetNarayanaPols}).
Finally, in Section \ref{sec:ADT}, we give a new proof of the Aztec diamond theorem
based on the discussions in the foregoing sections about LBPs, Schr\"oder paths and the $q$-Narayana polynomials.
Section \ref{sec:conclusions} is devoted to concluding remarks.

\section{Laurent biorthogonal polynomials and T-fractions}
\label{sec:LBPs_TFrac}

In Section \ref{sec:LBPs_TFrac}, we recall the definition and the fundamentals of LBPs and T-fractions.
See, e.g., \cite{Jones-Thron(1982),Hendriksen-VanRossum(1986),Zhedanov(1998)} for more details.
The formulations of LBPs may differ depending on the authors though they are essentially equivalent.
In this paper, we adopt the formulation in \cite{Zhedanov(1998)}.

\subsection{Laurent biorthogonal polynomials}
\label{sec:LBPs}

Let $b_{n+1}$ and $c_{n}$ for $n \in \mathbb{N}$ be arbitrary nonzero constants.
The (monic) {\em Laurent biorthogonal polynomials (LBPs)} $P_n(z)$, $n \in \mathbb{N}$, is the polynomials
determined from the recurrence
\begin{gather} \label{eq:recurrence}
  P_{n+1}(z) = (z - c_n) P_n(z) - b_n z P_{n-1}(z) \qquad \text{for $n \ge 1$}
\end{gather}
with the initial values $P_{0}(z) = 1$ and $P_1(z) = z - c_0$.
The first few of the LBPs are
\begin{subequations}
  \begin{align}
    P_0(z) &= 1, \\
    P_1(z) &= z - c_0, \\
    P_2(z) &= z^2 - (b_1 + c_0 + c_1) z + c_0 c_1, \\
    P_3(z) &= z^3 - (b_1 + b_2 + c_0 + c_1 + c_2) z^2 \nonumber \\
           & \qquad\qquad {} + (b_1 c_2 + b_2 c_0 + c_0 c_1 + c_0 c_2 + c_1 c_2) z - c_0 c_1 c_2.
  \end{align}
\end{subequations}
The LBP $P_n(z)$ is a monic polynomial in $z$ exactly of degree $n$ of which the constant term does not vanish.
In fact,
\begin{gather} \label{eq:LBPConstTerm}
  P_n(0) = (-1)^{n} \prod_{j=0}^{n-1} c_j \neq 0.
\end{gather}

The orthogonality of LBPs is described in the following theorem, that is sometimes referred to
by {\em Favard type theorem}.

\begin{thm}[Favard type theorem for LBPs] \label{thm:Favard}
  There exists a linear functional $\mathcal{F}$ defined over Laurent polynomials in $z$
  with respect to which the LBPs $P_n(z)$ satisfy the orthogonality
  \begin{gather} \label{eq:orthty}
    \mathcal{F}[P_n(z) z^{-k}] = h_n \delta_{n,k} \qquad \text{for $0 \le k \le n$}
  \end{gather}
  with some constants $h_n \neq 0$, where $\delta_{n,k}$ denotes the Kronecker delta.
  The linear functional $\mathcal{F}$ is unique up to a constant factor.
\end{thm}

We can prove Theorem \ref{thm:Favard} in almost the same way as Favard's theorem for orthogonal polynomials.
See, e.g., Chihara's book \cite[Chapter I, Theorem 4.4]{Chihara(1978OP)}.

We write the {\em moments} of the linear functional $\mathcal{F}$,
\begin{gather} \label{eq:mom}
  f_k = \mathcal{F}[z^k], \qquad k \in \mathbb{Z}.
\end{gather}
We fix the first moment $f_1 = \mathcal{F}[z]$ by
\begin{gather}
  f_1 = \kappa
\end{gather}
where $\kappa$ is an arbitrary nonzero constant.
We can show that the {\em moment determinant} of T\"oplitz form
\begin{gather}
  \Delta^{(s)}_{n} = \det(f_{s-j+k})_{j,k=0,\ldots,n-1} = {}
  \begin{vmatrix}
    f_{s}     & f_{s+1}   & \cdots & f_{s+n-1} \\
    f_{s-1}   & f_{s}     & \cdots & f_{s+n-2} \\
    \vdots    & \vdots    & \ddots & \vdots    \\
    f_{s-n+1} & f_{s-n+2} & \cdots & f_{s}     \\
  \end{vmatrix}
\end{gather}
does not vanish for $s \in {\{ 0, 1 \}}$ and $n \in \mathbb{N}$.
The LBPs $P_n(z)$ have the determinant expression
\begin{gather}
  P_n(z) = \frac{1}{\Delta^{(0)}_{n}}
  \begin{vmatrix}
    f_0      & f_1      & \cdots & f_{n-1} & f_n     \\
    f_{-1}   & f_0      & \cdots & f_{n-2} & f_{n-1} \\
    \vdots   & \vdots   & \ddots & \vdots  & \vdots  \\
    f_{-n+1} & f_{-n+2} & \cdots & f_0     & f_1     \\
    1        & z        & \cdots & z^{n-1} & z^n     \\
  \end{vmatrix}.
\end{gather}
Thus, from \eqref{eq:recurrence} and \eqref{eq:orthty},
the coefficients $b_n$ and $c_n$ of the recurrence \eqref{eq:recurrence}
and the constants $h_n$ in the orthogonality \eqref{eq:orthty}
\begin{gather} \label{eq:cfsInDets}
  b_n = -\frac{\Delta^{(1)}_{n+1} \Delta^{(0)}_{n-1}}{\Delta^{(1)}_{n} \Delta^{(0)}_{n}}, \qquad
  c_n = \frac{\Delta^{(1)}_{n+1} \Delta^{(0)}_{n}}{\Delta^{(1)}_{n} \Delta^{(0)}_{n+1}}, \qquad
  h_n = \frac{\Delta^{(0)}_{n+1}}{\Delta^{(0)}_{n}}.
\end{gather}

The inverted polynomials
\begin{gather}
  \tilde{P}_n(z) = \frac{z^{n} P_n(z^{-1})}{P_n(0)}
\end{gather}
also make a family of LBPs which are determined by the recurrence \eqref{eq:recurrence}
with the different coefficients
\begin{gather} \label{eq:bc_arginv}
  \tilde{b}_{n} = \frac{b_n}{c_{n-1} c_{n}}, \qquad
  \tilde{c}_{n} = \frac{1}{c_n}.
\end{gather}
We can determine a linear functional $\tilde{\mathcal{F}}$ for $\tilde{P}_n(z)$ by the moments
\begin{gather} \label{eq:momDual}
  \tilde{f}_k = \tilde{\mathcal{F}}[z^k] = f_{1-k}, \qquad k \in \mathbb{Z}.
\end{gather}
Then,
\begin{gather} \label{eq:kappa-kappaTilde}
  \tilde{f}_{1} = \tilde{\kappa} := f_{0} = \frac{\kappa}{c_0}.
\end{gather}

\begin{subequations} \label{eq:detsInCfs}
  The equations \eqref{eq:cfsInDets} and \eqref{eq:bc_arginv} imply that
  \begin{align}
  \Delta^{(1)}_{n} &= {}
    (-1)^{\frac{n(n-1)}{2}} \kappa^{n}
    \prod_{k=1}^{n-1} {\left( \frac{b_{k}}{c_{k-1}} \right)}^{n-k},
    \label{eq:detsInCfs01} \\
    \Delta^{(0)}_{n} &= {}
    (-1)^{\frac{n(n-1)}{2}} \tilde{\kappa}^{n}
    \prod_{k=1}^{n-1} {\left( \frac{\tilde{b}_{k}}{\tilde{c}_{k-1}} \right)^{n-k}}.
    \label{eq:detsInCfs00}
  \end{align}
\end{subequations}
In Section \ref{sec:ToeplitzDets}, we make use of the formulae \eqref{eq:detsInCfs}
to compute the moment determinant $\Delta^{(s)}_{n}$.

\subsection{T-fractions}

A {\em T-fraction} is a continued fraction
\begin{gather} \label{eq:T-fraction}
  T(z) = \CFrac{\kappa}{z - c_0} - \CFrac{b_1 z}{z - c_1} - \CFrac{b_2 z}{z - c_2} - \cdots.
\end{gather}
The $n$-th convergent of $T(z)$
\begin{gather}
  T_n(z) = \CFrac{\kappa}{z - c_0} - \CFrac{b_1 z}{z - c_1} - \cdots - \CFrac{b_{n-1} z}{z - c_{n-1}}
\end{gather}
is expressed by a ratio of polynomials
\begin{gather}
  T_n(z) = \frac{Q_n(z)}{P_n(z)}
\end{gather}
where $P_n(z)$ is the LBP of degree $n$ determined by the recurrence \eqref{eq:recurrence},
and $Q_n(z)$ is the polynomial determined by the same recurrence \eqref{eq:recurrence}
from different initial values $Q_0(z) = 0$ and $Q_1(z) = \kappa$.
Thus, we can identify the LBP $P_n(z)$ with the denominator polynomial of $T_n(z)$.

In Pad\'e approximants, the convergent $T_n(z)$ simultaneously approximates two formal power series
\begin{gather} \label{eq:momSeries}
  F_{+}(z) = \sum_{k=1}^{\infty} f_k z^{-k}
  \qquad \text{and} \qquad
  F_{-}(z) = -\sum_{k=0}^{\infty} f_{-k} z^k
\end{gather}
in the sense that
\begin{subequations} \label{eq:PadeApprox}
  \begin{align}
    T_n(z) &= F_{+}(z) + \mathop{\mathrm{O}}(z^{-n-1}) && \text{as $z \to \infty;$} \\[1\jot]
        {} &= F_{-}(z) + \mathop{\mathrm{O}}(z^{n})    && \text{as $z \to 0,$}
  \end{align}
\end{subequations}
where $f_k = \mathcal{F}[z^k]$ are the moments of the LBPs $P_n(z)$.
That is, expanded into series at $z=\infty$ and at $z=0$,
$T_n(z) = Q_n(z)/P_n(z)$ coincide with $F_+(z)$ and $F_-(z)$, respectively,  at least in the first $n$ terms.
The approximation \eqref{eq:PadeApprox} of $F_{+}(z)$ and $F_{-}(z)$ by $T_n(z)$ is equivalent
to the orthogonality \eqref{eq:orthty} of LBPs.

Taking the limit $n \to \infty$ in \eqref{eq:PadeApprox},
we observe that the T-fraction $T(z)$ equals to $F_{+}(z)$ and $F_{-}(z)$ as formal power series,
\begin{subequations} \label{eq:TFracSeries}
  \begin{align}
    T(z) &= F_{+}(z) && \text{as $z \to \infty;$} \\[1\jot]
      {} &= F_{-}(z) && \text{as $z \to 0.$}
  \end{align}
\end{subequations}

\section{Moments and Schr\"oder paths}
\label{sec:momPaths}

In Section \ref{sec:momPaths}, we give a combinatorial interpretation to the moments of LBPs.
Theorem \ref{thm:momPaths} of expressing each moment in terms of Schr\"oder paths is already shown 
in \cite[Theorem 8]{Kamioka(2007)}.
In this paper, we review the result by providing two new simple proofs.
The lattice path interpretation of LBPs is quite analogous
to those in the combinatorial interpretation of orthogonal polynomials by Viennot \cite{Viennot(1983OP)}.
We owe the idea of the proof in Section \ref{sec:proofTFrac} by T-fractions
to a combinatorial interpretation of continued fractions by Flajolet \cite{Flajolet(1980)}.

Let $P$ be a Schr\"oder path.
We label each step in $P$ by unity if the step is an up step,
by $b_n$ if a down step descending from the line $y=n$
and by $c_n$ if a level step on the line $y=n$,
where $b_n$ and $c_n$ are the coefficients of the recurrence \eqref{eq:recurrence} of the LBPs $P_n(z)$.
We then define the {\em weight} $w(P)$ of $P$ by the product of the labels of all the steps in $P$.
For example, the path in Figure \ref{fig:SchPath} weighs $w(P) = b_1^2 b_2^3 b_3 c_0 c_1 c_2^2$.
In the same way, labeling each step in $P$
using the recurrence coefficients $\tilde{b}_{n}$ and $\tilde{c}_{n}$ for $\tilde{P}_n(z)$,
we define another weight $\tilde{w}(P)$.
The main statement in Section \ref{sec:momPaths} is the following.

\begin{thm} \label{thm:momPaths}
  The moments $f_k = \mathcal{F}[z^k]$ of LBPs admit the expressions
  \begin{subequations} \label{eq:momentsPaths}
    \begin{align}
      f_{k+1} &= \kappa \sum_{P \in S_k} w(P),
      \label{eq:momPathsPos} \\
      f_{-k} &= \tilde{\kappa} \sum_{P \in S_k} \tilde{w}(P) && \text{for $k \in \mathbb{N}$.}
      \label{eq:momPathsNeg}
    \end{align}
  \end{subequations}
\end{thm}

For example,
\begin{subequations}
  \begin{align}
    f_{-2} &= \tilde{\kappa} (\tilde{b}_1 \tilde{b}_2 + \tilde{b}_1^2 + \tilde{b}_1 \tilde{c}_1 + {}
    2 \tilde{b}_1 \tilde{c}_0 + \tilde{c}_0^2), \\
    f_{-1} &= \tilde{\kappa} (\tilde{b}_1 + \tilde{c}_0), \\
     f_{0} &= \tilde{\kappa}, \\
     f_{1} &= \kappa, \\
     f_{2} &= \kappa (b_1 + c_0), \\
     f_{3} &= \kappa (b_1 b_2 + b_1^2 + b_1 c_1 + 2 b_1 c_0 + c_0^2).
  \end{align}
\end{subequations}

In the rest of Section \ref{sec:momPaths}, we show two different proofs of Theorem \ref{thm:momPaths}.
The first proof in Section \ref{sec:proofLBPs} is based on LBPs.
The second proof in Section \ref{sec:proofTFrac} is based on T-fractions.

\subsection{Proof of Theorem \ref{thm:momPaths} by LBPs}
\label{sec:proofLBPs}

\begin{lem} \label{lem:momPathsGen}
  For $n \in \mathbb{N}$ and $k \in \mathbb{N}$,
  \begin{subequations} \label{eq:fvodsjcb}
    \begin{align}
      \mathcal{F}[P_{n}(z) z^{k+1}] &= \kappa \sum_{P} w(P),
      \label{eq:momPathsGen01} \\
      \tilde{\mathcal{F}}[\tilde{P}_{n}(z) z^{k+1}] &= \tilde{\kappa} \sum_{P} \tilde{w}(P)
      \label{eq:momPathsGen00}
    \end{align}
  \end{subequations}
  where both the sums range over all Schr\"oder paths $P$ from $(-n,-n)$ to $(2k,0)$.
\end{lem}

\begin{proof}
  Let us write $f_{n,k} = \mathcal{F}[P_n(z) z^{k+1}]$.
  From the recurrence \eqref{eq:recurrence} of $P_n(z)$, we obtain a recurrence of $f_{n,k}$
  \begin{gather} \label{eq:uwcondvvp}
    f_{n,k} = f_{n+1,k-1} + c_n f_{n,k-1} + b_n f_{n-1,k}
  \end{gather}
  for $n \in \mathbb{N}$ and $k \in \mathbb{N}$,
  where the boundary values $f_{-1,k} = 0$ and $f_{n,-1} = \tilde{\kappa} \delta_{n,0}$ are induced
  from \eqref{eq:orthty} and \eqref{eq:kappa-kappaTilde}.
  The recurrence \eqref{eq:uwcondvvp} leads us to a combinatorial expression of \eqref{eq:momPathsGen01},
  \begin{gather}
    f_{n,k} = \tilde{\kappa} c_0 \sum_{P} w(P) = \kappa \sum_{P} w(P)
  \end{gather}
  where the sum ranges over all Schr\"oder paths $P$ from $(-n,-n)$ to $(2k,0)$.
  In much the same way, we can derive \eqref{eq:momPathsGen00}
  using Schr\"oder paths labelled with $\tilde{b}_{n}$ and $\tilde{c}_{n}$.
\end{proof}

From \eqref{eq:mom} and \eqref{eq:momDual},
Theorem \ref{thm:momPaths} is the special case of $n=0$ in Lemma \ref{lem:momPathsGen}.
That completes the proof of Theorem \ref{thm:momPaths} by LBPs.

\subsection{Proof of Theorem \ref{thm:momPaths} by T-fractions}
\label{sec:proofTFrac}

For a Schr\"oder path $P$, we define $\mathrm{length}(P)$ by
the sum of half the number of up and down steps and the number of level steps in $P$.
For example, the path $P$ in Figure \ref{fig:SchPath} is as long as $\mathrm{length}(P) = 10$.

\begin{lem} \label{lem:TFracPaths}
  The T-fraction $T(z)$ admits the expansions into formal power series
  \begin{subequations} \label{eq:TFracPaths}
    \begin{align}
      T(z) &= \kappa \sum_{P} w(P) z^{-\mathrm{length}(P)-1}               && \text{as $z \to \infty;$}
      \label{eq:TFracPathsInfty} \\
        {} &= -\tilde{\kappa} \sum_{P} \tilde{w}(P) z^{\mathrm{length}(P)} && \text{as $z \to 0,$}
      \label{eq:TFracPathsZero}
    \end{align}
  \end{subequations}
  where the both (formal) sums range over all Schr\"oder paths $P$ from $(0,0)$ to some point on the $x$-axis.
\end{lem}

\begin{proof}
  Let us consider partial convergents of $T(z)$
  \begin{gather}
    T_{m,n}(z) = \CFrac{1}{z - c_m} - \CFrac{b_{m+1} z}{z - c_{m+1}} - \cdots - \CFrac{b_{m+n-1} z}{z - c_{m+n-1}}
  \end{gather}
  for $m \in \mathbb{N}$ and $n \in \mathbb{N}$, where $T_{m,0}(z) = 0$.
  We first show by induction for $n \in \mathbb{N}$
  that $T_{m,n}(z) = S_{m,n}(z)$ as $n \to \infty$ where $S_{m,n}(z)$ denotes the formal power series
  \begin{gather} \label{eq:uig8vweu}
    S_{m,n}(z) = \sum_{P} w(P) z^{-\mathrm{length}(P)-1}
  \end{gather}
  over all Schr\"oder paths $P$ from $(m,m)$ to some point on the line $y=m$
  which lie in the region bounded by $y=m$ and $y=m+n-1$.
  (Hence, all the points in $P$ have the $y$-coordinates $\ge m$ and $\le m+n-1$.)
  For $n=0$, it is trivial that $S_{m,0}(z) = 0$ because the region in which $P$ may live is empty.
  Hence, $T_{m,0}(z) = S_{m,0}(z) = 0$.

  Suppose that $n \ge 1$.
  We classify Schr\"oder paths $P$ in the sum \eqref{eq:uig8vweu} into three classes:
  (i) the empty path only of one point at $(m,m)$ (without steps) of weight $1$;
  (ii) paths $P_2$ beginning by an up step;
  (iii) paths $P_3$ beginning by a level step.
  Thus,
  \begin{gather} \label{eq:nsiuhgbve}
    S_{m,n}(z) = {}
    z^{-1} + \sum_{P_2} w(P_2) z^{-\mathrm{length}(P_2)-1} + \sum_{P_3} w(P_3) z^{-\mathrm{length}(P_3)-1},
  \end{gather}
  where the sums with respect to $P_2$ and $P_3$ are taken
  over all Schr\"oder paths in the classes (ii) and (iii), respectively.
  Each path $P_2$ in the class (ii) consists of an initial level step on $y=m$, labelled $c_m$,
  and a subpath (maybe empty) from $(m+2,m)$ to some point on $y=m$.
  Hence,
  \begin{gather} \label{eq:osh89ewvh}
    \sum_{P_2} w(P_2) z^{-\mathrm{length}(P_2)-1} = c_m z^{-1} S_{m,n}(z).
  \end{gather}
  Each path $P_3$ in the class (iii), as shown in Figure \ref{fig:pathDecomp},
  uniquely decomposed into four parts:
  (A) an initial up step, labelled unity;
  (B) a subpath (maybe empty) from $(m+1,m+1)$ to some point on $y=m+1$ never going beneath $y=m+1$;
  (C) the first down step descending from $y=m+1$ to $y=m$, labelled $b_{m+1}$;
  (D) a subpath (maybe empty) both of whose initial and terminal points are on $y=m$.
  \begin{figure}
    \centering
\includegraphics{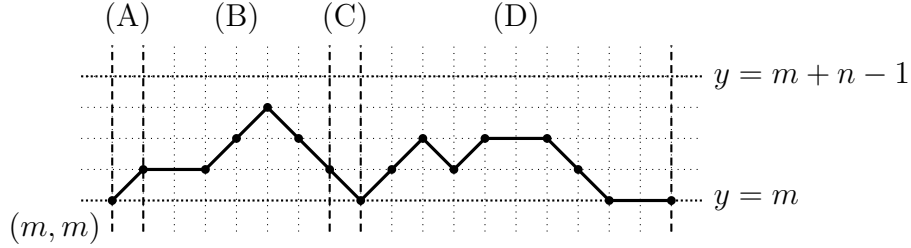}
    \caption{The decomposition of a path $P_3$ in the class (iii) into four parts (A), (B), (C) and (D).}
    \label{fig:pathDecomp}
  \end{figure}
  Hence,
  \begin{gather} \label{eq:nuch87fvf}
    \sum_{P_3} w(P_3) z^{-\mathrm{length}(P_3)-1} = b_{m+1} S_{m+1,n-1}(z) S_{m,n}(z).
  \end{gather}
  Substituting \eqref{eq:osh89ewvh} and \eqref{eq:nuch87fvf} into \eqref{eq:nsiuhgbve}, we get
  \begin{gather}
    S_{m,n}(z) = {\{ z - c_m - b_{m+1} z S_{m+1,n-1}(z) \}^{-1}}.
  \end{gather}
  From the assumption of induction,
  we can assume that $S_{m+1,n-1}(z) = T_{m+1,n-1}(z)$ as $z \to \infty$ and hence
  \begin{gather}
    S_{m,n}(z) = {\{ z - c_m - b_{m+1} z T_{m+1,n-1}(z) \}^{-1}} = T_{m,n}(z) \qquad \text{as $z \to \infty.$}
  \end{gather}

  Now, let us prove Lemma \ref{lem:TFracPaths}.
  In taking the limit $n \to \infty$ of the identity $T_{0,n}(z) = S_{0,n}(z)$ as $z \to \infty$,
  the left-hand side $T_{0,n}(z)$ tends to $T(z)$
  while the right-hand side $S_{0,n}$ to the right-hand side of \eqref{eq:TFracPathsInfty}.
  In order to show \eqref{eq:TFracPathsZero}, we observe from \eqref{eq:bc_arginv}
  that $T_{m,n}(z)$ is equivalent to
  \begin{gather}
    T_{m,n}(z) = {}
    -\CFrac{\tilde{c}_{m} z^{-1}}{z^{-1} - \tilde{c}_{m}} - {}
    \CFrac{\tilde{b}_{m+1} z^{-1}}{z^{-1} - \tilde{c}_{m+1}} - \cdots - {}
    \CFrac{\tilde{b}_{m+n-1} z^{-1}}{z^{-1} - \tilde{c}_{m+n-1}}.
  \end{gather}
  We can thereby show \eqref{eq:TFracPathsZero} as a simple corollary of \eqref{eq:TFracPathsInfty}.
  That completes the proof of Lemma \ref{lem:TFracPaths}.
\end{proof}

The expressions \eqref{eq:momentsPaths} of moments in Theorem \ref{thm:momPaths} are derived
just by equating \eqref{eq:TFracSeries} and \eqref{eq:TFracPaths}.
Indeed, every Schr\"oder path $P$ from $(0,0)$ to some point on the $x$-axis terminates at $(2k,0)$
if and only if $\mathrm{length}(P) = k$.
That completes the proof of Theorem \ref{thm:momPaths} by T-fractions.

\section{Non-intersecting Schr\"oder paths}
\label{sec:nonIntPaths}

In Section \ref{sec:nonIntPaths}, as a consequence of Theorem \ref{thm:momPaths},
we examine the moment determinant $\Delta^{(s)}_{n}$ from a combinatorial viewpoint.
We utilize {\em Gessel--Viennot's lemma} \cite{Gessel-Viennot(1985), Aigner(2001LNCS)}
to read the determinant in terms of non-intersecting paths.

For $m \in \mathbb{N}$ and $n \in \mathbb{N}$,
let $\bm{S}_{m,n}$ denote the set of $n$-tuples $\bm{P} = (P_0,\ldots,P_{n-1})$ of Schr\"oder paths $P_k$
such that (i) $P_{k}$ goes from $(-k,k)$ to $(2m+k,k)$
and that (ii) every two distinct paths $P_j$ and $P_k$, $j \neq k$, are {\em non-intersecting},
namely $P_j \cap P_k = \emptyset$.
As shown in Figure \ref{fig:nonIntPaths},
each $n$-tuple $\bm{P} \in \bm{S}_{m,n}$ can be drawn in a diagram of $n$ non-intersecting Schr\"oder paths
which are pairwise disjoint.
\begin{figure}
  \centering
\includegraphics{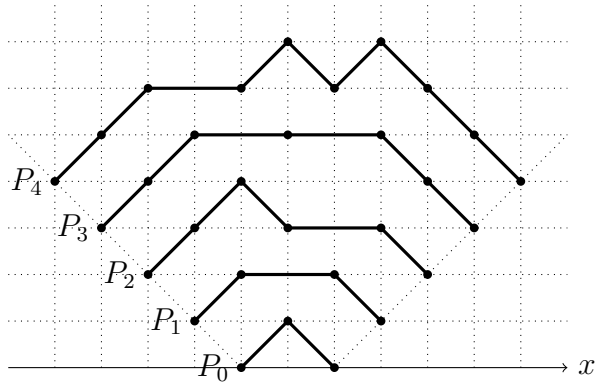}
  \caption{%
    A quintuple $\bm{P} = (P_0,\ldots,P_4) \in \bm{S}_{1,5}$ of non-intersecting Schr\"oder paths
    which is drawn in a plane.
  }
  \label{fig:nonIntPaths}
\end{figure}
For simplicity, we write
\begin{gather}
  w(\bm{P}) = \prod_{k=0}^{n-1} w(P_k), \qquad
  \tilde{w}(\bm{P}) = \prod_{k=0}^{n-1} \tilde{w}(P_k).
\end{gather}

\begin{thm} \label{thm:detInNIPaths}
  \begin{subequations}
    For general $b_n$ and $c_n$ nonzero, the moment determinant $\Delta^{(s)}_{n}$ admits the expressions
    \begin{align}
      \Delta^{(s)}_{n}
      {} &= (-1)^{\frac{n(n-1)}{2}} \kappa^{n} {\left( \prod_{j=1}^{n-1} b_j^{n-j} \right)}
      \sum_{\bm{P} \in \bm{S}_{s-n,n}} w(\bm{P})
      && \text{if $s \ge n;$}
      \label{eq:GV} \\[2\jot]
      {} &= (-1)^{\frac{n(n-1)}{2}} \tilde{\kappa}^{n} {\left( \prod_{j=1}^{n-1} \tilde{b}_j^{n-j} \right)}
      \sum_{\bm{P} \in \bm{S}_{|s|-n+1,n}} \tilde{w}(\bm{P})
      && \text{if $s \le -n+1.$}
      \label{eq:GVDual}
    \end{align}
  \end{subequations}
\end{thm}

\begin{proof}
  Suppose that $s \ge n \ge 0$.
  We rewrite $\Delta^{(s)}_{n}$ into Hankel form,
  \begin{gather} \label{eq:gqof8e9}
    \Delta^{(s)}_{n} = (-1)^{\frac{n(n-1)}{2}} \det(f_{s-n+j+k+1})_{j,k=0,\ldots,n-1}.
  \end{gather}
  Owing to Theorem \ref{thm:momPaths},
  the $(j,k)$-entry of the last Hankel determinant has the combinatorial expression
  \begin{gather}
    f_{s-n+j+k+1} = \kappa \sum_{P_{j,k}} w(P_{j,k})
  \end{gather}
  where we can assume that the sum ranges over all Schr\"oder paths $P_{j,k}$ from $(-2j,0)$ to $(2(s-n)+2k,0)$.
  Thus, we can apply Gessel--Viennot's lemma \cite{Gessel-Viennot(1985),Aigner(2001LNCS)}
  to expand the determinant \eqref{eq:gqof8e9},
  \begin{gather} \label{eq:voihvedavj}
    \det(f_{s-n+j+k+1})_{j,k=0,\ldots,n-1} = {}
    \kappa^{n} \sum_{(P_{0,0},\ldots,P_{n-1.n-1})} w(P_{0,0}) \cdots w(P_{n-1,n-1})
  \end{gather}
  where the sum ranges over all $n$-tuples $(P_{0,0},\ldots,P_{n-1,n-1})$
  of {\em non-intersecting} Schr\"oder paths $P_{k,k}$
  such that $P_{k,k}$ goes from $(-2k,0)$ to $(2(s-n)+2k,0)$ for each $k$.
  (See Figure \ref{fig:gcd9dlcwv} for example.)
  \begin{figure}
    \centering
\includegraphics{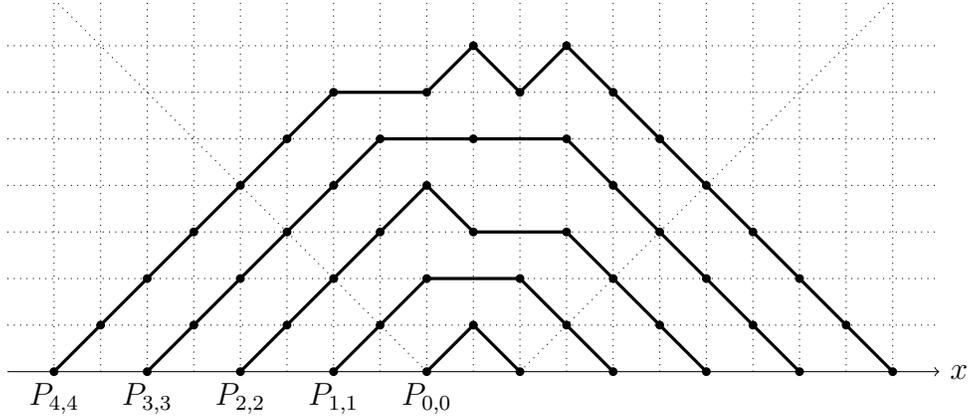}
    \caption{%
      A quintuple $(P_{0,0},\ldots,P_{4,4}) $ of non-intersecting Schr\"oder paths counted
      in the right-hand sum of \eqref{eq:voihvedavj} ($m=1$ and $n=5$).
    }
    \label{fig:gcd9dlcwv}
  \end{figure}
  As shown in Figure \ref{fig:gcd9dlcwv},
  the first and last $k$ steps of $P_{k,k}$ must be all up and down steps, respectively,
  so that the paths do not collide.
  Especially, $P_{k,k}$ passes the points $(-k,k)$ and $(2(s-n)+k,k)$.
  We thus have
  \begin{gather}
    \det(f_{s-n+j+k+1})_{j,k=0,\ldots,n-1} = {}
    \kappa^{n} {\left( \prod_{j=1}^{n-1} b_j^{n-j} \right)} \sum_{\bm{P} \in \bm{S}_{s-n,n}} w(\bm{P})
  \end{gather}
  and thereby \eqref{eq:GV}.
  In the same way, we can show \eqref{eq:GVDual} from Theorem \ref{thm:momPaths}.
\end{proof}

For example, for $m=3$ and $n=2$, the set $\bm{S}_{1,2}$ contains exactly eight doubles $(P_0,P_1)$
of non-intersecting Schr\"oder paths which are shown in Figure \ref{fig:NIPaths}.
\begin{figure}
  \centering
\includegraphics{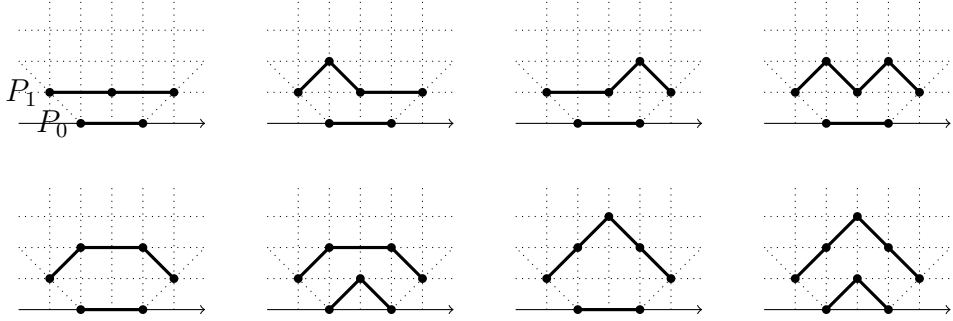}
  \caption{%
    The eight doubles $(P_0,P_1) \in \bm{S}_{1,2}$ of non-intersecting Schr\"oder paths.
  }
  \label{fig:NIPaths}
\end{figure}
Thus, the moment determinant $\Delta^{(3)}_{2}$ equals to the polynomial of eight monomials
\begin{gather}
  \Delta^{(3)}_{2} = {}
  -\kappa^2 b_1 (c_0 c_1^2 + 2 b_2 c_0 c_1 + b_2^2 c_0 + b_2 c_0 c_2 + b_1 b_2 c_2 + b_2 b_3 c_0 + b_1 b_2 b_3)
\end{gather}
of which each monomial corresponds to a diagram in Figure \ref{fig:NIPaths}.

\section{$q$-Narayana polynomials as moments}
\label{sec:qNarPolMom}

In Section \ref{sec:qNarPolMom} and the subsequent,
we consider the special case of the LBPs whose moments are described by the $q$-Narayana polynomials.
Let us recall from Section \ref{sec:introduction} the definition of the $q$-Narayana polynomials
\begin{gather} \label{eq:qNaraPolRp}
  N_{k}(t,q) = \sum_{P \in S_k} t^{\mathrm{level}(P)} q^{\mathrm{area}(P)}
\end{gather}
where $\mathrm{level}(P)$ denotes the number of level steps in a Schr\"oder path $P$,
and $\mathrm{area}(P)$ the area bordered by $P$ and the $x$-axis.
For example, the first few of the $q$-Narayana polynomials are enumerated in
\begin{subequations}
  \begin{align}
    N_0(t,q) &= 1, \\
    N_1(t,q) &= t+q, \\
    N_2(t,q) &= t^2 + 2 t q + t q^3 + q^2 + q^4, \\
    N_3(t,q) &= t^3 + 3 t^2 q + 2 t^2 q^3 + t^2 q^5 + 3 t q^2 + 4 t q^4 + 2 t q^6 + t q^8 \nonumber \\
    & \qquad {} + q^3 + 2 q^5 + q^7 + q^9.
  \end{align}
\end{subequations}

In view of Theorem \ref{thm:momPaths}, it is easy to find the $q$-Narayana polynomials in the moments of LBPs.

\begin{thm} \label{thm:NarayanaPolsInMoms}
  Let us determine the LBPs $P_{n}(z)$ by the recurrence \eqref{eq:recurrence} with the coefficients
  \begin{gather} \label{eq:dsach87we7}
    b_{n} = q^{2n-1}, \qquad
    c_{n} = t q^{2n}.
  \end{gather}
  Then,
  \begin{subequations} \label{eq:NarayanaPolsInMoms}
    the linear functional $\mathcal{F}$ for $P_{n}(z)$ admits
    the moments $f_k = \mathcal{F}[z^k]$ described by the $q$-Narayana polynomials,
    \begin{align}
      f_{k}
      {} &= \kappa N_{k-1}(t,q)              && \text{for $k \ge 1;$}
      \label{eq:momNarPos} \\[1\jot]
      {} &= \kappa t^{-2|k|-1} N_{|k|}(t,q^{-1}) && \text{for $k \le 0,$}
      \label{eq:momNarNeg}
    \end{align}
    where $\kappa$ is an arbitrary nonzero constant.
  \end{subequations}
\end{thm}

\begin{proof}
  Let $P \in S_k$.
  Labelled with \eqref{eq:dsach87we7}, $P$ weighs $w(P) = t^{\mathrm{level}(P)} q^{\mathrm{area}(P)}$.
  Hence, by virtue of Theorem \ref{thm:momPaths}, we have \eqref{eq:momNarPos}
  as a special case of \eqref{eq:momPathsPos}.
  Similarly, with
  \begin{gather} \label{eq:NarayanaCfsDual}
    \tilde{b}_{n} = t^{-2} q^{-2n+1}, \qquad
    \tilde{c}_{n} = t^{-1} q^{-2n}
  \end{gather}
  from \eqref{eq:bc_arginv}, $\tilde{w}(P) = t^{-2k+\mathrm{level}(P)} q^{-\mathrm{area}(P)}$.
  Now $\tilde{\kappa} = \kappa t^{-1}$ from \eqref{eq:kappa-kappaTilde}.
  Hence, we obtain \eqref{eq:momNarNeg} from \eqref{eq:momPathsNeg}.
\end{proof}

We remark that the $q$-Narayana polynomials \eqref{eq:qNaraPolRp} defined in a combinatorial way are
already investigated by Cigler \cite{Cigler(2005arXiv)}
who introduced the polynomials by modifying the generating function of the $q$-Catalan numbers.
Indeed, we can deduce from \eqref{eq:qNaraPolRp} a recurrence
\begin{gather} \label{eq:qNarRec}
  N_k(t,q) = t N_{k-1}(t,q) + \sum_{j=0}^{k-1} q^{2j+1} N_j(t,q) N_{k-j-1}(t,q).
\end{gather}
We can identify \eqref{eq:qNarRec} with the recurrence in \cite[Eq.~(19)]{Cigler(2005arXiv)}.

\section{Determinant of $q$-Narayana polynomials}
\label{sec:ToeplitzDets}

In Section \ref{sec:ToeplitzDets}, we examine a T\"oplitz determinant of the $q$-Narayana polynomials
\begin{gather} \label{eq:DetNarayanaPol}
  \mathcal{N}^{(s)}_{n}(t,q) = \det(N_{s+j-k-1}(t,q))_{j,k=0,\ldots,n-1},
\end{gather}
where, in view of \eqref{eq:NarayanaPolsInMoms}, we define $N_k(t,q)$ for negative $k$ by
\begin{gather}
  N_{k}(t,q) = t^{-2|k|-1} N_{|k|-1}(t,q^{-1}) \qquad \text{for $k < 0$.}
\end{gather}
As the special case of the moments given by the $q$-Narayana polynomials,
Theorem \ref{thm:detInNIPaths} allows us to read the determinant \eqref{eq:DetNarayanaPol}
in the context of non-intersecting Schr\"oder paths.
The results in this section, Theorem \ref{thm:DetNarayanaPols} and Corollary \ref{cor:NIPathsNarayana},
will be applied later in Section \ref{sec:ADT} to a new proof of the Aztec diamond theorem (Theorem \ref{thm:ADT}).

Theorem \ref{thm:NarayanaPolsInMoms} implies that
\begin{gather} \label{eq:hiusgf7e}
  \Delta^{(s)}_{n} = \mathcal{N}^{(s)}_{n}(t,q)
\end{gather}
provided that the coefficients $b_n$ and $c_n$ of the recurrence \eqref{eq:recurrence} are given
by \eqref{eq:dsach87we7} where $\kappa = 1$.
Hence, we can use the formulae \eqref{eq:detsInCfs}
to find the exact value of $\mathcal{N}^{(s)}_{n}$ for $s \in {\{ 0, 1 \}}$.
Recall that, in using \eqref{eq:detsInCfs00},
we assume $\tilde{b}_{n}$ and $\tilde{c}_{n}$ to be given by \eqref{eq:NarayanaCfsDual}
and $\tilde{\kappa} = t^{-1}$.

\begin{lem} \label{lem:NarayanaDets}
  For $s \in {\{ 0, 1 \}}$ and $n \in \mathbb{N}$,
  the exact value of the determinant $\mathcal{N}^{(s)}_{n}$ is given by
  \begin{subequations} \label{eq:NarayanaDets}
    \begin{align}
      \mathcal{N}^{(1)}_{n}(t,q) &= (-1)^{\frac{n(n-1)}{2}} t^{-\frac{n(n-1)}{2}} q^{\frac{n(n-1)}{2}},
      \label{eq:NarayanaDets_one} \\[1\jot]
      \mathcal{N}^{(0)}_{n}(t,q) &= (-1)^{\frac{n(n-1)}{2}} t^{-\frac{n(n+1)}{2}} q^{-\frac{n(n-1)}{2}}.
      \label{eq:NarayanaDets_zero}
    \end{align}
  \end{subequations}
\end{lem}

In order to find the value of $\mathcal{N}^{(s)}_{n}(t,q)$ for further $s \in \mathbb{Z}$ and $n \in \mathbb{N}$,
we can use Sylvester's determinant identity:
\begin{gather} \label{eq:SylvesterDetId}
  X \cdot X(i,j;k,\ell) - X(i;k) \cdot X(j;\ell) + X(i;\ell) \cdot X(j;k) = 0
\end{gather}
where $X$ is an arbitrary determinant and
$X(i,j;k,\ell)$ denotes the minor of $X$
obtained by deleting the $i$-th and the $j$-th rows and the $k$-th and the $\ell$-th columns;
$X(i;k)$ the minor of $X$ with respect to the $i$-th row and the $k$-th column.
Applying Sylvester's determinant identity, we get
\begin{gather} \label{eq:NarayanaSylvester}
  \mathcal{N}^{(s)}_{n+1} \cdot \mathcal{N}^{(s)}_{n-1} - {}
  \mathcal{N}^{(s)}_{n} \cdot \mathcal{N}^{(s)}_{n} + {}
  \mathcal{N}^{(s+1)}_{n} \cdot \mathcal{N}^{(s-1)}_{n} = 0
\end{gather}
for $s \in \mathbb{Z}$ and $n \in \mathbb{N}$,
where $\mathcal{N}^{(s)}_{n} = \mathcal{N}^{(s)}_{n}(t,q)$ except that $\mathcal{N}^{(s)}_{-1} = 0$.
Using \eqref{eq:NarayanaSylvester} as a recurrence from appropriate initial value,
we can compute the value of $\mathcal{N}^{(s)}_{n}(t,q)$ for each $s \in \mathbb{Z}$ and $n \in \mathbb{N}$.
Especially, we find a closed form of $\mathcal{N}^{(s)}_{n}(t,q)$ for $-n \le s \le n+1$ as follows.

\begin{thm} \label{thm:DetNarayanaPols}
  For $-n \le s \le n+1$,
  the exact value of the determinant $\mathcal{N}^{(s)}_{n}(t,q)$ is given by
  \begin{subequations} \label{eq:NarayanaDetsCor}
    \begin{align}
      \mathcal{N}^{(s)}_{n}(t,q)
      {} &= \varphi^{(s)}_{n}(t,q) \prod_{k=1}^{s-1} (t+q^{2k-1})^{s-k} &&
      \text{for $1 \le s \le n+1;$}
      \label{eq:DetNarayanaPols+} \\
      {} &= \varphi^{(s)}_{n}(t,q) \prod_{k=1}^{|s|} (t+q^{-2k+1})^{|s|-k+1} &&
      \text{for $-n \le s \le 0$}
    \end{align}
    where
    \begin{gather}
      \varphi^{(s)}_{n}(t,q) = (-1)^{\frac{n(n-1)}{2}} t^{-\frac{(n-s)(n-s+1)}{2}} q^{\frac{n(n-1)(2s-1)}{2}}.
    \end{gather}
  \end{subequations}
\end{thm}

\begin{proof}
  Using Sylvester's identity \eqref{eq:NarayanaSylvester} from the initial value \eqref{eq:NarayanaDets},
  we can easily show \eqref{eq:NarayanaDetsCor} by induction.
\end{proof}

Note that Cigler \cite{Cigler(2005arXiv)} found a closed-form expression
of the Hankel determinant $\det(N_{s+j+k}(t,q))_{j,k=0,\ldots,n-1}$ of the $q$-Narayana polynomials
for $s \in {\{ 0, 1 \}}$ and $n \in \mathbb{N}$ by means of orthogonal polynomials.
(The Hankel determinant coincides with $\mathcal{N}^{(s)}_{n}(t,q)$ for $s \in {\{ n, n+1 \}}$ without sign.)
Theorem \ref{thm:DetNarayanaPols} generalizes Cigler's result \cite[Eqs.~(24) and (25)]{Cigler(2005arXiv)}
for further $s$ and $n$.

As a corollary of Theorem \ref{thm:DetNarayanaPols},
equating \eqref{eq:DetNarayanaPols+} with \eqref{eq:GV} in Theorem \ref{thm:detInNIPaths},
we obtain the following result about non-intersecting Schr\"oder paths.

\begin{cor} \label{cor:NIPathsNarayana}
  For $m \in {\{ 0,1 \}}$ and $n \in \mathbb{N}$,
  \begin{gather}
    \sum_{\bm{P} \in \bm{S}_{m,n}} t^{\mathrm{level}(\bm{P})} q^{\mathrm{area}(\bm{P})} = {}
    q^{\frac{n(n-1)(3m+2n-1)}{3}} \prod_{k=1}^{m+n-1} (t + q^{2k-1})^{m+n-k}
  \end{gather}
  where $\mathrm{level}(\bm{P}) = \sum_{k=0}^{n-1} \mathrm{level}(P_k)$
  and $\mathrm{area}(\bm{P}) = \sum_{k=0}^{n-1} \mathrm{area}(P_k)$ with $\bm{P} = (P_0,\ldots,P_{n-1})$.
\end{cor}

\section{Proof of Aztec diamond theorem}
\label{sec:ADT}

Finally, in Section \ref{sec:ADT}, we give a new proof of the Aztec diamond theorem (Theorem \ref{thm:ADT})
based on the discussions in the foregoing sections.
In the two-parted paper by Elkies, Kuperberg, Larsen and Propp
\cite{Elkies-Kuperberg-Larsen-Propp(1992.01),Elkies-Kuperberg-Larsen-Propp(1992.02)},
the Aztec diamond theorem is proven by the technique of the domino shuffling.
Whereas, the proof in this paper is based on the one-to-one correspondence
between tilings of the Aztec diamonds and tuples of non-intersecting Schr\"oder paths
developed by Eu and Fu \cite{Eu-Fu(2005)} who used the correspondence to prove \eqref{eq:ADT11}.

In order to make the statement precise, as we announced in Section \ref{sec:introduction},
we review from \cite{Elkies-Kuperberg-Larsen-Propp(1992.01)} the definition of the rank statistic.
Let $T \in T_n$ be a tiling of the Aztec diamond $\mathit{AD}_n$.
If $n \ge 1$, $T$ certainly contains one or more two-by-two blocks of two horizontal or vertical dominoes.
Thus, choosing one from such two-by-two blocks and rotating it by ninety degrees,
we obtain a new tiling $T' \in T_n$.
(See Figure \ref{fig:elmMove}.)
\begin{figure}
  \centering
\includegraphics{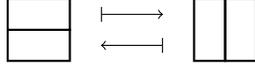}
  \caption{Rotation of a two-by-two block of two horizontal or vertical dominoes in an elementary move.}
  \label{fig:elmMove}
\end{figure}
We refer by an {\em elementary move} to this operation of transforming $T$ into $T'$ by rotating a two-by-two block.
It can be shown that any tiling of $\mathit{AD}_n$ can be reached from any other tiling of $\mathit{AD}_n$
by a sequence of elementary moves.
The rank $r(T)$ of $T$ denotes the minimal number of elementary moves
required to reach $T$ from the ``all-horizontal'' tiling $T^0$ consisting only of horizontal dominoes,
where $r(T^0) = 0$.
For example, in Figure \ref{fig:rank}, the rightmost tiling $T$ of $\mathit{AD}_{2}$ has the rank $r(T) = 4$
since at least four elementary moves are required to reach from the leftmost $T^0$.
\begin{figure}
  \centering
\includegraphics{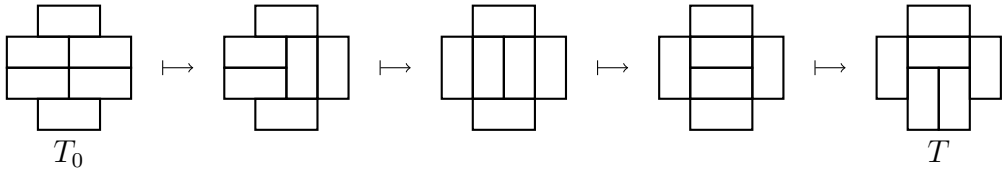}
  \caption{%
    A sequence of elementary moves from $T^0$ to $T$ of $\mathit{AD}_{2}$.
    At least four elementary moves are required to reach from $T^0$ only of horizontal dominoes
    to the rightmost $T$, and thereby $r(T) = 4$.
  }
  \label{fig:rank}
\end{figure}

Eu and Fu \cite{Eu-Fu(2005)} developed a one-to-one correspondence between $T_n$ and $\bm{S}_{1,n}$.
We describe the bijection from $T_n$ to $\bm{S}_{1,n}$ in a slightly different manner from \cite{Eu-Fu(2005)}.
Following \cite{Elkies-Kuperberg-Larsen-Propp(1992.01)},
we color the Aztec diamond $\mathit{AD}_n$ in a black-white checkerboard fashion
so that all unit squares on the upper-left border of $\mathit{AD}_n$ are white.
We say that a horizontal domino (resp.~a vertical domino) put into $\mathit{AD}_n$ is {\em even}
if the left half (resp.~the upper half) of the domino covers a white unit square.
Otherwise, the domino is {\em odd}.
The bijection mapping a tiling $T \in T_n$
to an $n$-tuple $\bm{P} = (P_0,\ldots,P_{n-1}) \in \bm{S}_{1,n}$ of non-intersecting Schr\"oder paths is described
by the following procedure:
For each domino in $T$, as shown in Figure \ref{fig:step-domino},
draw an up step (resp.~ a down step, a level step) that goes through the center of the domino
if the domino is even vertical (resp.~odd vertical, odd horizontal).
(For even horizontal dominoes, we do nothing.)
\begin{figure}
  \centering
\includegraphics{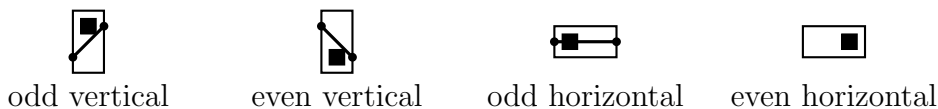}
  \caption{The rule to draw a step on a domino.}
  \label{fig:step-domino}
\end{figure}
Then, we find $n$ non-intersecting Schr\"oder paths $P_0,\ldots,P_{n-1}$ on $T$
of which the $n$-tuple $\bm{P} = (P_0,\ldots,P_{n-1})$ belongs to $\bm{S}_{1,n}$.
For example, see Figure \ref{fig:bijection}.
\begin{figure}
  \centering
\includegraphics{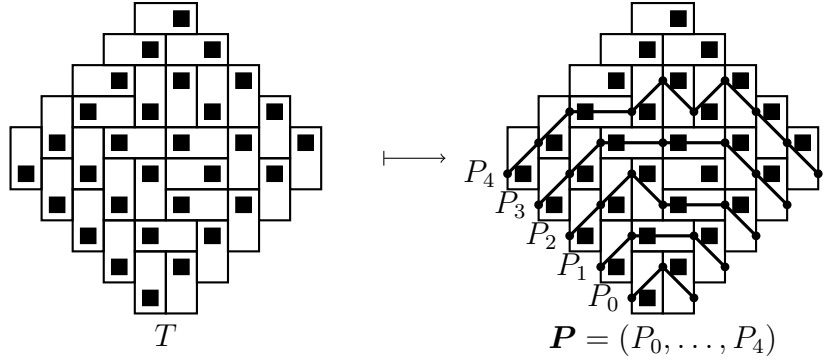}
  \caption{%
    The bijection mapping a tiling $T \in T_5$ of $\mathit{AD}_5$
    to a quintuple $\bm{P} = (P_0,\ldots,P_4) \in S_{1,5}$ of non-intersecting Schr\"oder paths.
    (The Aztec diamond is colored in a checkerboard fashion.)
  }
  \label{fig:bijection}
\end{figure}

The bijection connects the statistics $v(T)$ and $r(T)$ for tilings
and the statistics $\mathrm{level}(P)$ and $\mathrm{area}(P)$ for Schr\"oder paths as follows.
Recall from Section \ref{sec:introduction}
that $v(T)$ denotes half the number of vertical dominoes in a tiling $T$.

\begin{lem} \label{lem:1-1}
  Suppose that a tiling $T \in T_n$
  and an $n$-tuple $\bm{P} = (P_0,\ldots,P_{n-1}) \in \bm{S}(1,n)$ of non-intersecting Schr\"oder paths are
  in the one-to-one correspondence by the bijection.
  Then,
  \begin{align}
    v(T) &= \frac{n(n+1)}{2} - \mathrm{level}(\bm{P}),
    \label{eq:1-1_vert-level} \\
    r(T) &= \mathrm{area}(\bm{P}) - \frac{2n(n+1)(n-1)}{3},
    \label{eq:1-1_rank-area}
  \end{align}
  where $\mathrm{level}(\bm{P}) = \sum_{k=0}^{n-1} \mathrm{level}(P_k)$
  and $\mathrm{area}(\bm{P}) = \sum_{k=0}^{n-1} \mathrm{area}(P_k)$.
\end{lem}

\begin{proof}
  The bijection implies that $v(T)$ equals to half the number of up and down steps in $\bm{P}$.
  The sum of half the number of up and down steps and the number of level steps in $\bm{P}$ is a constant
  independent of $\bm{P}$ that equals to $n(n+1)/2$.
  Thus, we have \eqref{eq:1-1_vert-level}.

  As shown in Figure \ref{fig:emoves-pathDefms},
  each elementary move of a tiling $T$ raising the rank by one gives rise
  to a deformation of some path in $\bm{P}$ increasing the area by one.
  \begin{figure}
    \centering
\includegraphics{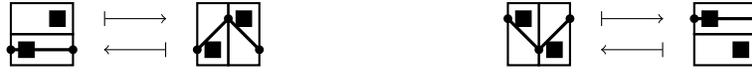}
    \caption{%
      A rotation of a two-by-two block in an elementary move raises the rank of the tiling $T$ by one
      (left-to-right, respectively)
      if and only if the corresponding deformation of a tuple $\bm{P}$
      of non-intersecting Schr\"oder paths increases $\mathrm{area}(\bm{P})$ by one.
    }
    \label{fig:emoves-pathDefms}
  \end{figure}
  Thus, $r(T)$ and $\mathrm{area}(\bm{P})$ differ by a constant independent of $T$ and $\bm{P}$.
  Since $r(T^0) = 0$ then the constant equals to $\mathrm{area}(\bm{P}^0) = 2n(n+1)(n-1)/3$,
  where $T^0$ denotes the ``all-horizontal'' tiling  of $\mathit{AD}_n$
  and $\bm{P}^0 \in \bm{S}_{1,n}$ the $n$-tuple of non-intersecting Schr\"oder paths only of level steps
  that corresponds to $T^0$.
  Thus, we have \eqref{eq:1-1_rank-area}.
\end{proof}

Now, we give a proof of the Aztec diamond theorem.

\begin{proof}[Proof of Theorem \ref{thm:ADT}]
  As a consequence of Lemma \ref{lem:1-1},
  we can substitute \eqref{eq:1-1_vert-level} and \eqref{eq:1-1_rank-area} into \eqref{eq:ADT} to obtain
  \begin{gather} \label{eq:lnciubewg9}
    \mathrm{AD}_n(t,q) = {}
    t^{\frac{n(n+1)}{2}} q^{-\frac{2n(n-1)(n+1)}{3}}
    \sum_{\bm{P} \in S_{1,n}} t^{-\mathrm{level}(\bm{P})} q^{\mathrm{area}(\bm{P})}.
  \end{gather}
  From Corollary \ref{cor:NIPathsNarayana},
  the sum in the right-hand side of \eqref{eq:lnciubewg9} is equated with
  \begin{gather} \label{eq:o0qkpvsw}
    \sum_{\bm{P} \in S_{1,n}} t^{-\mathrm{level}(\bm{P})} q^{\mathrm{area}(\bm{P})} = {}
    q^{\frac{2n(n-1)(n+1)}{3}} \prod_{k=1}^{n} (t^{-1} + q^{2k-1})^{n-k+1}.
  \end{gather}
  Substituting \eqref{eq:o0qkpvsw} into the right-hand side of \eqref{eq:lnciubewg9}, we have
  \begin{gather}
    \mathrm{AD}_n(t,q) = {}
    t^{\frac{n(n+1)}{2}} \prod_{k=1}^{n} (t^{-1} + q^{2k-1})^{n-k+1} = {}
    \prod_{k=1}^{n} (1 + t q^{2k-1})^{n-k+1}.
  \end{gather}
  That completes the proof of Theorem \ref{thm:ADT}.
\end{proof}

\section{Concluding remarks}
\label{sec:conclusions}

In this paper, we evaluated a determinant whose entries are given by the $q$-Narayana polynomials
(Theorem \ref{thm:DetNarayanaPols}).
In order to find the value of the determinant, we utilized Laurent biorthogonal polynomials
which allow of a combinatorial interpretation in terms of Schr\"oder paths
(Theorem \ref{thm:momPaths} and Theorem \ref{thm:NarayanaPolsInMoms}).
As an application, we exhibited a new proof of the Aztec diamond theorem (Theorem \ref{thm:ADT})
by Elkies, Kuperberg, Larsen and Propp
\cite{Elkies-Kuperberg-Larsen-Propp(1992.01),Elkies-Kuperberg-Larsen-Propp(1992.02)}
with the help of the one-to-one correspondence developed by Eu and Fu \cite{Eu-Fu(2005)}
between tilings of the Aztec diamonds and tuples of non-intersecting Schr\"oder paths.

We remark that, in Theorem \ref{thm:DetNarayanaPols},
we can evaluate the determinant $\mathcal{N}^{(s)}_{n}$ of the $q$-Narayana polynomials
also for $s < -n$ and $s > n+1$ by using the formula \eqref{eq:NarayanaSylvester} from Sylvester's identity.
For example, if $s = n+2$,
\begin{gather} \label{eq:whf8escnp}
  \mathcal{N}^{(n+2)}_{n} = {}
  (-1)^{\frac{n(n-1)}{2}} q^{\frac{n(n-1)(2n+3)}{2}} \prod_{k=1}^{n} (t + q^{2k-1})^{n-k+1}
  \sum_{\ell=0}^{n} t^{n-\ell} q^{\ell^2} \binom{n+1}{\ell}_{q^2},
\end{gather}
where $\binom{m}{n}_{q}$ denotes the $q$-binomial coefficient
\begin{gather}
  \binom{m}{n}_{q} = \prod_{k=1}^{n} \frac{1-q^{m-k+1}}{1-q^{k}}.
\end{gather}
From Theorem \ref{thm:detInNIPaths},
we can read \eqref{eq:whf8escnp} in terms of non-intersecting Schr\"oder paths,
\begin{multline}
  \sum_{\bm{P} \in \bm{S}_{2,n}} t^{\mathrm{level}(\bm{P})} q^{\mathrm{area}(\bm{P})} \\
  {} = {}
  q^{\frac{n(n-1)(2n+5)}{3}} \prod_{k=1}^{n} (t + q^{2k-1})^{n-k+1}
  \sum_{\ell=0}^{n} t^{n-\ell} q^{\ell^2} \binom{n+1}{\ell}_{q^2}.
\end{multline}

We can readily observe
that the bijection in Section \ref{sec:ADT} gives a one-to-one correspondence
between $n$-tuples of non-intersecting Schr\"oder paths in $\bm{S}_{2,n}$
and tilings of the region $\mathit{AD}_{2,n}$, 
the Aztec diamond $\mathit{AD}_{n+1}$ from which two unit squares at the south corner are removed.
(See Figure \ref{fig:pathsTiling25}).
\begin{figure}
  \centering
\includegraphics{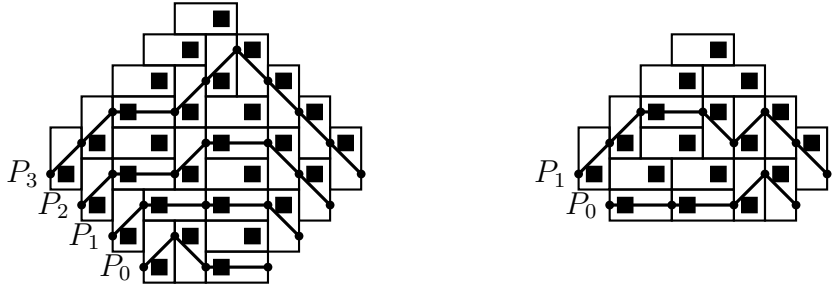}
  \caption{%
    The one-to-one correspondence between a tiling of $\mathit{AD}_{m,n}$
    and a $n$-tuple $\bm{P} = (P_0,\ldots,P_{n-1}) \in \bm{S}_{m,n}$ of non-intersecting Schr\"oder paths.
    (The left figure shows an instance for $m=2$ and $n=4$ while the right for $m=3$ and $n=2$.)
  }
  \label{fig:pathsTiling25}
\end{figure}
Therefore, as a variant of \eqref{eq:ADT}, we have 
\begin{gather}
  \sum_{T} t^{v(T)} q^{r(T)} = {}
  \prod_{k=0}^{n-1} (1 + t q^{2k+1})^{n-k} \sum_{\ell=0}^{n} t^{\ell} q^{\ell^2} \binom{n+1}{\ell}_{q^2},
\end{gather}
where the sum in the left-hand side ranges over all tilings $T$ of $\mathit{AD}_{2,n}$.
(The $\mathrm{rank}(T)$ is defined in the same way as $\mathit{AD}_{n}$
to be the minimal number of elementary moves
required to reach from ``all-horizontal'' tilings of $\mathit{AD}_{2,n}$.)
Similarly, calculating the determinant $\mathcal{N}^{(m+n)}_{n}(t,q)$,
we can obtain in principle variant formulae of \eqref{eq:ADT} 
for tilings of the Aztec diamond $\mathit{AD}_{m+n}$
from which $m(m-1)$ unit squares at the south corner are removed.
However, the value of $\mathcal{N}^{(m+n)}_{n}(t,q)$ seems much complicated for large $m$,
and exact formulae has not been found yet for general $m$ and $n$.

\section*{Acknowledgment}

The author would like to thank Professor Yoshimasa Nakamura and Professor Satoshi Tsujimoto
for valuable discussions and comments.
This work was supported by JSPS KAKENHI 24740059.





\bibliographystyle{model1b-num-names}
\bibliography{ksh94}







\end{document}